\documentclass[reqno,oneside]{amsart}

\usepackage[english]{babel} 
\usepackage[latin1]{inputenc} 
\usepackage[T1]{fontenc} 
\usepackage{lmodern}
\usepackage{amsmath,amsthm,amssymb,amsfonts,mathrsfs,latexsym} 
\usepackage[colorlinks, linkcolor=blue, citecolor=blue, urlcolor=blue, hypertexnames=true]{hyperref}  %
\usepackage[pagewise]{lineno}

\makeatletter
\newtheorem*{rep@theorem}{\rep@title}
\newcommand{\newreptheorem}[2]{%
\newenvironment{rep#1}[1]{%
 \def\rep@title{#2 \ref{##1}}%
 \begin{rep@theorem}}%
 {\end{rep@theorem}}}
\makeatother

\newtheorem{thm}[equation]{Theorem}
\newreptheorem{thm}{Theorem}
\newreptheorem{cor}{Corollary}
\newtheorem{lem}[equation]{Lemma}
\newtheorem{prop}[equation]{Proposition}
\newtheorem{cor}[equation]{Corollary}
\newtheorem*{thm*}{Theorem}

\theoremstyle{definition}

\newtheorem{exam}[equation]{Example}

\newtheorem{que}{Question}

\newcommand\mb\mathbb
\newcommand\mf\mathfrak
\newcommand\mr\mathrm

\newcommand{\C}{\mathbb{C}}

\newcommand{\R}{\mathbb{R}}
\newcommand{\Z}{\mathbb{Z}}

\renewcommand\epsilon\varepsilon
\renewcommand\phi\varphi
\renewcommand\rho\varrho
\renewcommand\tilde\widetilde
\renewcommand\hat\widehat
\renewcommand\bar\overline

\DeclareMathOperator{\SL}{SL}
\DeclareMathOperator{\SO}{SO}
\DeclareMathOperator{\SU}{SU}
\DeclareMathOperator{\Sp}{Sp}
\DeclareMathOperator{\PSL}{PSL}
\DeclareMathOperator{\Alt}{Alt}

\newcommand{\WA}{\mr{WA}}
\renewcommand{\WA}{{}}
\newcommand{\dd}{\mr d}

\numberwithin{equation}{section}
\begin{document}
\selectlanguage{english} 

\begin{abstract}
We completely characterize connected Lie groups all of whose countable subgroups are weakly amenable.
We also provide a characterization of connected semisimple Lie groups that are weakly amenable.
Finally, we show that a connected Lie group is weakly amenable if the group is weakly amenable as a discrete group.
\end{abstract}


\title{
Weak amenability of Lie groups made discrete
}
\author{S{\o}ren Knudby}
\address{Mathematisches Institut der WWU M\"unster,
\newline Einsteinstra\ss{}e 62, 48149 M\"unster, Germany.}
\thanks{Supported by the Deutsche Forschungsgemeinschaft through the Collaborative Research Centre (SFB 878).}
\date{\today}
\maketitle

\section{Statement of the results}
Weak amenability for locally compact groups was introduced by Cowling and Haagerup in \cite{MR996553}. The property has proven useful as a tool in operator algebras going back to Haagerup's result on the free groups \cite{MR520930}, results on lattices on simple Lie groups and their group von Neumann algebras \cite{MR996553,MR3476201}, and more recently in several results on Cartan rigidity in the theory of von Neumann algebras (see e.g. \cite{MR3179609,MR3259044}). Due to its many applications in operator algebras, the study of weak amenability, especially for discrete groups, is important.

A locally compact group $G$ is weakly amenable if the constant function 1 on $G$ can be approximated uniformly on compact subsets by compactly supported Herz-Schur multipliers, uniformly bounded in norm (see Section~\ref{sec:prelim} for details). The optimal uniform norm bound is the \emph{Cowling--Haagerup constant} (or the \emph{weak amenability constant}), denoted here $\Lambda_\WA(G)$. 

By now, weak amenability is quite well studied, especially in the setting of connected Lie groups. The combined work of \cite{MR748862,MR996553,MR784292,MR1245415,MR1418350,MR3476201,MR1079871} characterizes weak amenability for simple Lie groups. For partial results in the non-simple case, we refer to \cite{MR2132866,MR3542782}. We record the simple case here.

\begin{thm}[\cite{MR748862,MR996553,MR784292,MR1245415,MR1418350,MR3476201,MR1079871}]\label{thm:simple}
A connected simple Lie group $G$ is weakly amenable if and only if the real rank of $G$ is zero or one. In that case, the weak amenability constant is
\begin{linenomath*}\begin{align}\label{eq:WA}
\Lambda_{\WA}(G) = \begin{cases}
1 & \text { when } G \text{ has real rank zero}, \\
1 & \text{ when } G\approx\SO(1,n),\ n\geq 2, \\
1 & \text{ when } G\approx\SU(1,n),\ n\geq 2, \\
2n-1 & \text{ when } G\approx\Sp(1,n),\ n\geq 2, \\
21 & \text{ when } G\approx \mr F_{4(-20)}. \\
\end{cases}
\end{align}\end{linenomath*}
Above, $G \approx H$ means that $G$ is locally isomorphic to $H$.
\end{thm}

In Section~\ref{sec:semisimple}, we observe how the classification of simple Lie groups that are weakly amenable can be extended to include all semisimple Lie groups. Since it is not known in general if weak amenability of connected Lie groups is preserved under local isomorphism, it is not entirely obvious how to deduce the semisimple case from the simple case. We prove the following:

\begin{repthm}{thm:semisimple}
Let $G$ be a connected semisimple Lie group. Then $G$ is locally isomorphic to a direct product $S_1\times\cdots\times S_n$ of connected simple Lie groups $S_1,\ldots,S_n$, and $G$ is weakly amenable if and only if each $S_i$ is weakly amenable. In fact,
\[
\Lambda_\WA(G) = \prod_{i=1}^n \Lambda_\WA(S_i).
\]
\end{repthm}

Combining Theorem~\ref{thm:semisimple} and Theorem~\ref{thm:simple} one can then compute the value $\Lambda_\WA(G)$ for any connected semisimple Lie group $G$. Our proof of Theorem~\ref{thm:semisimple} relies on an inequality proved by Cowling for discrete groups (see \eqref{eq:mod-central}). In order to apply Cowling's inequality, we pass to lattices by using Haagerup's result \eqref{eq:lattice} that this does not change the Cowling-Haagerup constant. The same trick is also used in our proof of Theorem~\ref{thm:non-discrete}.

Theorem~\ref{thm:semisimple} was previously known under the additional assumption that the semisimple Lie group had finite center or finite fundamental group. Indeed, under this additional assumption Theorem~\ref{thm:semisimple} then follows from Theorem~\ref{thm:simple} and an application of the well-known permanence properties \eqref{eq:mod-compact} and \eqref{eq:product} below. The assumption of finite center or finite fundamental group can be considered as extreme cases, and Theorem~\ref{thm:semisimple} then settles the intermediate cases in which the center and the fundamental group are both infinite.

In a similar spirit, the characterization of weak amenability for connected Lie groups in general has been done in the cases where the Levi factor has finite center \cite{MR2132866} or the Lie group has trivial fundamental group \cite{MR3542782} (the case of finite fundamental group then follows from \eqref{eq:mod-compact}). Some intermediate cases remain open.

For a locally compact group $G$, we let $G_\dd$ denote the same group $G$ equipped with the discrete topology. There has previously been some interest in studying the relationship between properties of $G$ and $G_\dd$ (see e.g. \cite{MR1181157,MR1219692,cornulier-jlt,MR3345044}). For instance, it is known that if $G_\dd$ is an amenable group, then $G$ is an amenable group. The analogous question about weak amenability is open: 
\begin{que}[\cite{MR3345044}]\label{que:18}
If $G_\dd$ is weakly amenable, is $G$ weakly amenable?
\end{que}

It is a fact that a discrete group is weakly amenable if and only if all of its countable subgroups are weakly amenable (see Lemma~\ref{lem:countable}).
It thus makes no difference if one studies weak amenability of $G_\dd$ or of all countable subgroups of $G$. Note that countable subgroups of $G$ are always viewed with the discrete topology which might differ from the subspace topology coming from $G$.

Our main result is the following characterization of connected Lie groups all of whose countable subgroups are weakly amenable.

\begin{repthm}{thm:discrete}
Let $G$ be a connected Lie group, and let $G_\dd$ denote the group $G$ equipped with the discrete topology. The following are equivalent.
\begin{enumerate}
	\item[(1)] $G$ is locally isomorphic to $R\times\SO(3)^a\times\SL(2,\R)^b\times\SL(2,\C)^c$, for a solvable connected Lie group $R$ and integers $a,b,c$.
	\item[(2)] $G_\dd$ is weakly amenable with constant 1.
	\item[(3)] $G_\dd$ is weakly amenable.
	\item[(4)] Every countable subgroup of $G$ is weakly amenable with constant $1$.
	\item[(5)] Every countable subgroup of $G$ is weakly amenable.
\end{enumerate}
\end{repthm}

In \cite{MR3345044}, Theorem~\ref{thm:discrete} was proved in the special case where $G$ is a simple Lie group. In order to remove the assumption of simplicity, one needs to deal with certain semidirect products, some of which were dealt with in \cite{MR3542782}. In Section~\ref{sec:discrete} we obtain non-weak amenability results for the remaining semidirect products (see Proposition~\ref{prop:four-cases}) and thus obtain Theorem~\ref{thm:discrete}.

Our proof of Theorem~\ref{thm:discrete} relies in part on the methods of \cite{cornulier-jlt} where de~Cornulier proved that (1) in Theorem~\ref{thm:discrete} is equivalent to
\begin{enumerate}
	\item[(6)] $G_\dd$ has the Haagerup property.
\end{enumerate}


It was conjectured by Cowling (see \cite[p.~7]{MR1852148}) that a locally compact group $G$ satisfies $\Lambda_\WA(G) = 1$ if and only if $G$ has the Haagerup property. Although this is now known to be false in this generality (see \cite[Remark~2.13]{MR2680430}, \cite[Corollary~2]{MR2393636}), Theorem~\ref{thm:discrete} together with de~Cornulier's result \cite[Theorem~1.14]{cornulier-jlt} establishes Cowling's conjecture for connected Lie groups made discrete.

As another application of Theorem~\ref{thm:discrete}, we are able to settle Question~\ref{que:18} in the case of connected Lie groups. In the last section, we establish the following.

\begin{repcor}{cor:question18}
Let $G$ be a connected Lie group. If $G_\dd$ is weakly amenable, then $G$ is weakly amenable. In this case, $\Lambda_\WA(G_\dd) = \Lambda_\WA(G) = 1$.
\end{repcor}

We remark that our proof of Corollary~\ref{cor:question18} relies on the classification obtained in Theorem~\ref{thm:discrete}. It would be preferable to have a direct proof avoiding the classification.

\section{Preliminaries}\label{sec:prelim}

\subsection{Weak amenability}
Let $G$ be a locally compact group. A \emph{Herz-Schur multiplier} is a complex function $\phi$ on $G$ of the form $\phi(y^{-1}x) = \langle P(x),Q(y)\rangle$, where $P,Q\colon G\to \mathcal H$ are bounded continuous functions from $G$ to a Hilbert space $\mathcal H$ and $x,y\in G$. Note that $\phi$ is continuous and bounded by $\|P\|_\infty\|Q\|_\infty$. The Herz-Schur norm of $\phi$ is defined as
\[
\|\phi\|_{B_2} = \inf\{\|P\|_\infty\|Q\|_\infty\},
\]
where the infimum is taken over all $P,Q$ as above. With this norm and pointwise operations, the Herz-Schur multipliers form a unital Banach algebra.

The group $G$ is \emph{weakly amenable} if there is a net $(\phi_i)$ of compactly supported Herz-Schur multipliers converging to 1 uniformly on compact subsets of $G$ and satisfying $\sup_i\|\phi_i\|_{B_2} \leq C$ for some $C\geq 1$. The \emph{weak amenability constant} $\Lambda_\WA(G)$ is the infimum of those $C\geq 1$ for which such a net exists, with the understanding that $\Lambda_\WA(G) = \infty$ if $G$ is not weakly amenable. There are several equivalent definitions of weak amenability in the literature, see e.g. \cite[Proposition~1.1]{MR996553}. Weak amenability of groups should however not be confused with weak amenability of Banach algebras.

Weak amenability is preserved under several group constructions. We list here the known results needed later on and refer to \cite[Section~12.3]{MR2391387}, \cite[Section~III]{MR1120720}, \cite[Section~1]{MR996553}, \cite{MR3476201}, \cite{MR3310701} for proofs.
When $K$ is a compact normal subgrop of $G$,
\begin{linenomath*}\begin{align}\label{eq:mod-compact}
\Lambda_\WA(G) = \Lambda_\WA(G/K).
\end{align}\end{linenomath*}

If $(G_i)_{i\in I}$ is a directed collection of open subgroups in $G$ then
\begin{linenomath*}\begin{align}\label{eq:union}
\Lambda_\WA\left(\bigcup_{i\in I} G_i\right) = \sup_{i\in I} \Lambda_\WA(G_i).
\end{align}\end{linenomath*}

For two locally compact groups $G$ and $H$,
\begin{linenomath*}\begin{align}\label{eq:product}
\Lambda_\WA(G\times H)=\Lambda_\WA(G)\Lambda_\WA(H).
\end{align}\end{linenomath*}

If $G$ has a closed normal subgroup $N$ such that the quotient $G/N$ is amenable then
\begin{linenomath*}\begin{align}\label{eq:co-amenable}
\Lambda_\WA(N)= \Lambda_\WA(G).
\end{align}\end{linenomath*}
We remark that \eqref{eq:co-amenable} is stated in \cite{MR3310701} only for second countable groups, but it is not difficult to deduce the general statement from this and the Kakutani-Kodaira Theorem \cite[Theorem~8.7]{MR551496} using \eqref{eq:mod-compact} and \eqref{eq:union}.

Recall that a lattice $\Gamma$ in a locally compact group $G$ is a discrete subgroup such that the homogeneous space $G/\Gamma$ admits a $G$-invariant probability measure, where $G$ acts on $G/\Gamma$ by left translation. If $\Gamma$ is a lattice in a second countable, locally compact group $G$, then
\begin{linenomath*}\begin{align}\label{eq:lattice}
\Lambda_\WA(\Gamma)= \Lambda_\WA(G).
\end{align}\end{linenomath*}

When $Z$ is a central subgroup of a discrete group $G$ then
\begin{linenomath*}\begin{align}\label{eq:mod-central}
\Lambda_\WA(G) \leq \Lambda_\WA(G/Z).
\end{align}\end{linenomath*}

A remark on \eqref{eq:mod-central} is in order. Much work related to weak amenability for connected Lie groups would be significantly easier if \eqref{eq:mod-central} holds true for non-discrete groups $G$ as well. For instance, \cite{MR1079871} would then have been an immediate consequence of earlier work such as \cite{MR748862,MR996553}, and our Theorem~\ref{thm:semisimple} would also be an immediate consequence of earlier work. It would even be relatively easy to complete the characterization of weak amenability for connected Lie groups. Needless to say, we have not been able to generalize \eqref{eq:mod-central} to the non-discrete case so far. Sometimes, one can reduce the general case to the discrete case and then apply \eqref{eq:mod-central}. In the present paper, this is done using lattices as is most explicitly seen in the proof of Theorem~\ref{thm:semisimple}, but also in Theorem~\ref{thm:non-discrete}.

\begin{lem}\label{lem:countable}
Let $G$ be a discrete group. Then $G$ is weakly amenable if and only if every countable subgroup of $G$ is weakly amenable.
\end{lem}
\begin{proof}
Clearly, weak amenability of $G$ implies that every subgroup of $G$ is weakly amenable. Assume conversely that $G$ is not weakly amenable. We claim that $G$ contains a countable subgroup which is not weakly amenable. Since $G$ is the directed union of all its countable subgroups, it follows from \eqref{eq:union} that there is a sequence $G_1,G_2,\ldots$ of countable subgroups of $G$ such that $\Lambda_\WA(G_n) \geq n$. Let $G_\infty$ be the subgroup of $G$ generated by $G_1,G_2,\ldots$. Then $G_\infty$ is a countable subgroup of $G$ and $G_\infty$ is not weakly amenable. This completes the proof.
\end{proof}

\subsection{Structure of Lie groups}\label{sec:Levi}
Losely speaking, two Lie groups are locally isomorphic if they admit homeomorphic neighborhoods of the identity on which the group laws (here only partially defined) are identical. Equivalently, two Lie groups are locally isomorphic if and only if their Lie algebras are isomorphic (see \cite[Theorem~II.1.11]{MR514561}).

A connected Lie group $G$ has a simply connected covering $\tilde G$ which is a Lie group locally isomorphic to $G$ in such a way that the covering map is a group homomorphism. The kernel of the covering homomorphism is a discrete central subgroup of $\tilde G$. Conversely, any connected Lie group locally isomorphic to $G$ is a quotient of $\tilde G$ by a discrete central subgroup. For a discrete subgroup $N$ of the center $Z(\tilde G)$ of $\tilde G$, then the center of the quotient $\tilde G/N$ is precisely the quotient of the center $Z(\tilde G)/N$. See e.g. \cite[Chapter~II]{MR0015396} and \cite[Section~I.11]{MR1920389} for details.

Let $G$ be a connected Lie group with Lie algebra $\mf g$. Then $G$ admits a Levi decomposition $G = RS$. Here, $R$ is the solvable closed connected Lie subgroup of $G$ associated with the solvable radical of $\mf g$. The group $S$ is a semisimple connected Lie subgroup of $G$ associated with a (semisimple) Levi subalgebra $\mf s$ of $\mf g$. We refer to \cite[Section~3.18]{MR746308} and in particular \cite[Theorem~3.18.13]{MR746308} for details. The semisimple Lie algebra $\mf s$ splits as a direct sum $\mf s = \mf s_1\oplus\cdots\oplus\mf s_n$ of simple Lie algebras (for some $n\geq 0$), and if $S_i$ denotes the connected Lie subgroup of $G$ associated with the Lie subalgebra $\mf s_i$, then $S$ is locally isomorphic to the direct product $S_1\times\cdots\times S_n$.

\section{Weak amenability of semisimple Lie groups}\label{sec:semisimple}

The computation below of $\Lambda_\WA(G)$ for all semisimple Lie groups $G$ basically relies on three facts: the existence of lattices in semisimple Lie groups, the permanence results stated in Section~\ref{sec:prelim}, and most importantly that $\Lambda_\WA(G)$ is known for all simple Lie groups. 

\begin{thm}\label{thm:semisimple}
Let $G$ be a connected semisimple Lie group. Then $G$ is locally isomorphic to a direct product $S_1\times\cdots\times S_n$ of connected simple Lie groups $S_1,\ldots,S_n$, and $G$ is weakly amenable if and only if each $S_i$ is weakly amenable. In fact,
\[
\Lambda_\WA(G) = \prod_{i=1}^n \Lambda_\WA(S_i).
\]
\end{thm}

\begin{proof}
Let $Z$ denote the center of $G$, $\tilde G$ the universal cover of $G$, and $\bar G = G/Z$. By semisimplicity, $Z$ is discrete. The Lie algebra $\mf g$ of $G$ is a direct sum $\mf g = \mf s_1\oplus\cdots\oplus\mf s_n$ of simple Lie algebras. Let $\tilde S_i$ and $\bar S_i$ denote the analytic subgroups of $\tilde G$ and $\bar G$ corresponding to $\mf s_i$, respectively. Then we have the following direct product decompositions
\[
\tilde G = \prod_{i=1}^n \tilde S_i \qquad\text{and}\qquad \bar G = \prod_{i=1}^n \bar S_i.
\]
Let $\bar\Gamma$ be a lattice in $\bar G$ (a lattice exists by \cite[Theorem~14.1]{MR0507234}).
Consider the covering homomorphisms
\[
\tilde G\to G \qquad\text{and}\qquad G\to\bar G,
\]
and let $\Gamma$ be the lift of $\bar\Gamma$ to $G$, and let $\tilde\Gamma$  be the lift of $\Gamma$ to $\tilde G$. Then $\Gamma\leq G$ is a lattice, and $\tilde\Gamma\leq\tilde G$ is a lattice. Using \eqref{eq:lattice}, \eqref{eq:product}, and \eqref{eq:mod-central} we obtain
\begin{linenomath*}\begin{align*}
\Lambda_\WA(G) &= \Lambda_\WA(\Gamma) \leq \Lambda_\WA(\bar\Gamma) = \Lambda_\WA(\bar G) = \prod_{i=1}^n \Lambda_\WA(\bar S_i), \\
 \Lambda_\WA(G) & = \Lambda_\WA(\Gamma) \geq \Lambda_\WA(\tilde\Gamma) = \Lambda_\WA(\tilde G) =\prod_{i=1}^n \Lambda_\WA(\tilde S_i).
\end{align*}\end{linenomath*}

By Theorem~\ref{thm:simple}, we have $\Lambda_\WA(\bar S_i) = \Lambda_\WA(\tilde S_i)$ for every $i$, and this concludes the proof.
\end{proof}

\section{Weak amenability of Lie groups made discrete}\label{sec:discrete}

When $G$ is a Lie group we denote by $G_\dd$ the group $G$ equipped with the discrete topology. We recall \cite[Theorem~1.10]{MR3345044} which will be used in the proof of Theorem~\ref{thm:discrete}.
\begin{thm}[\cite{MR3345044}]\label{thm:simple-discrete}
For a connected simple Lie group $S$, the following are equivalent.
\begin{itemize}
	\item $S$ is locally isomorphic to $\SO(3)$, $\SL(2,\R)$, or $\SL(2,\C)$.
	\item $S_\dd$ is weakly amenable.
	\item $S_\dd$ is weakly amenable with constant 1.
\end{itemize}
\end{thm}

In order to generalize Theorem~\ref{thm:simple-discrete} to non-simple Lie groups we need to consider certain semidirect products which we now describe. A main ingredient to prove non-weak amenability of these semidirect products is \cite[Theorem~5]{MR3542782}, which we recall here for the convenience of the reader.
\begin{thm}[{\cite[Theorem~5]{MR3542782}}]\label{thm:ozawa}
Let $H\curvearrowright N$ be an action by automorphisms of a discrete group $H$ on a discrete group $N$, and let $G = N\rtimes H$ be the corresponding semidirect product group. Let $N_0$ be a proper subgroup of $N$. Suppose
\begin{enumerate}
	\item $H$ is not amenable;
	\item $N$ is amenable;
	\item $N_0$ is $H$-invariant;
	\item For every $x\in N\setminus N_0$, the stabilizer of $x$ in $H$ is amemable.
\end{enumerate}
Then $G$ is not weakly amenable.
\end{thm}

The semidirect products of interest also appear in \cite{cornulier-jlt} to which we refer the reader for further details.


The irreducible real representations of $\SL(2,\R)$ and $\SU(2)$ are well-known. We describe them below.

For each natural number $n\geq 1$, the group $\SL(2,\R)$ has a unique irreducible real representation $V_n$ of dimension $n$. (see \cite[p.~107]{MR0430163}).
It may be realized as the natural action of $\SL(2,\R)$ on the homogeneous polynomials in two real variables of degree $n-1$.

Similarly, the group $\SU(2)$ acts on the homogeneous polynomials in two complex variables of degree $n-1$. When $n = 2m$ is even, this representation is still irreducible as a real representation $V_{4m}$ of dimension $4m$. When $n = 2m+1$ is odd, the representation is the complexification of an irreducible real representation $V_{2m+1}$ of dimension $2m+1$. The representations $V_{2m+1}$ and $V_{4m}$ make up all the irreducible real representations of $\SU(2)$. We refer to \cite{MR781344,MR1090745} for details.

Let $S$ be $\SL(2,\R)$ or $\SU(2)$, and let $\mf s$ be the Lie algebra of $S$. If $V$ is a real irreducible representation of $S$, then $V$ also carries the derived representation of $\mf s$. Let $\Alt_{\mf s}(V)$ denote the real vector space of alternating bilinear forms $\phi$ on $V$ that are $\mf s$-invariant, that is, bilinear forms $\phi\colon V\times V\to\R$ satisfying
\[
\phi(x,x) = 0
\qquad\text{and}\qquad
\phi(s.x,y) + \phi(x,s.y) = 0 \qquad\text{for all } s\in\mf s,\ x,y\in V.
\]
The Lie group $H(V)$ is defined as $V\times\Alt_{\mf s}(V)^*$ with group multiplication given by
\[
(x,z)(x',z') = (x+x', z+z' + e_{x,x'}), \qquad x,x'\in V,\ z,z'\in\Alt_{\mf s}(V)^*,
\]
where $e_{x,x'}\in\Alt_{\mf s}(V)^*$ is the evaluation functional defined by $e_{x,x'}(\phi) = \phi(x,x')$. The group $S$ acts on $H(V)$ by $s.(x,z) = (s.x,z)$. When $Z\subseteq\Alt_{\mf s}(V)^*$ is a subspace, we obtain a quotient group $H(V)/Z$, and the action of $S$ on $H(V)$ descends to an action on $H(V)/Z$. In this way we obtain the semidirect product
\[
H(V)/Z\rtimes S.
\]
\begin{lem}\label{lem:proper}
If $G$ is a proper, real algebraic subgroup of $\SL(2,\R)$ or $\SU(2)$, then $G_\dd$ is amenable.
\end{lem}
\begin{proof}
Let $S$ be $\SL(2,\R)$ or $\SU(2)$, and let $\mf s$ be the Lie algebra of $S$. The group $G$ has only finitely many components (in the usual Hausdorff topology) (see \cite[Theorem~3]{MR0095844} or \cite[Theorem~3.6]{MR1278263}). It is therefore enough to show that the identity component $G_0$ of $G$ is amenable as a discrete group.

Since $G_0$ is a connected, proper, closed subgroup of $S$, its Lie algebra $\mf g$ is a proper Lie subalgebra of $\mf s$. The dimension of $\mf g$ is therefore at most two, and $\mf g$ must be a solvable Lie algebra. So $G_0$ is a solvable group. In particular, $G_0$ is amenable in the discrete topology.
\end{proof}

In what follows below, we have to exclude the trivial irreducible representation of $S$. We thus assume from now on that $\dim V \geq 2$.

\begin{lem}\label{lem:stabilizer}
If $\dim V\geq 2$ and if $(x,z)\in H(V)/Z$ and $x\neq 0$, then the stabilizer of $(x,z)$ in $S$ is amenable in the discrete topology.
\end{lem}
\begin{proof}
The stabilizer of $(x,z)$ in $S$ coincides with the stabilizer of $x$ in $S$. Since $x\neq 0$, and $S$ acts irreducibly on $V$, the stabilizer of $x$ in $S$ is a proper subgroup. It follows from the explicit description of the action of $S$ on $V$ as the action on the homogeneous polynomials in two variables that the stabilizer is moreover a real algebraic subgroup. Hence, Lemma~\ref{lem:proper} shows that the stabilizer is amenable in the discrete topology.
\end{proof}

\begin{prop}\label{prop:minimal-groups}
If $\dim V\geq 2$, the group $H(V)/Z\rtimes S$ is not weakly amenable in the discrete topology.
\end{prop}
\begin{proof}
We intend to apply Theorem~\ref{thm:ozawa} with $H = S_\dd$, $N = (H(V)/Z)_\dd$ and $N_0 = (\Alt_{\mf s}(V)^*/Z)_\dd$. Clearly, $H$ is not amenable, $N$ is amenable, and $N_0$ is invariant under $H$ (in fact, $H$ acts trivially on $N_0$). It remains to check that every element of $N\setminus N_0$ has amenable stabilizer. This is Lemma~\ref{lem:stabilizer}.
\end{proof}

In the case $S = \SL(2,\R)$, the group $H(V)/Z\rtimes S$ is not simply connected, since $\SL(2,\R)$ is not simply connected. Let $\tilde S = \tilde\SL(2,\R)$ denote the universal covering group of $\SL(2,\R)$. The covering homomorphism $\tilde S\to S$ has kernel isomorphic to the group of integers $\Z$. The group $\tilde S$ acts on $H(V)/Z$ through the action of $S$. Since the stabilizer of $(x,z)\in H(V)/Z$ in $\tilde S$ is an extension by $\Z$ of the stabilizer in $S$, and since amenability is preserved by extensions, the following is immediate from Lemma~\ref{lem:stabilizer}.
\begin{lem}
If $\dim V\geq 2$ and if $(x,z)\in H(V)/Z$ and $x\neq 0$, then the stabilizer of $(x,z)$ in $\tilde S$ is amenable in the discrete topology.
\end{lem}

Applying Theorem~\ref{thm:ozawa} with $H = \tilde S_\dd$, $N = (H(V)/Z)_\dd$ and $N_0 = (\Alt_{\mf s}(V)^*/Z)_\dd$, we obtain the following:
\begin{prop}\label{prop:SL2-covering}
If $\dim V\geq 2$, the group $H(V)/Z\rtimes\tilde S$ is not weakly amenable in the discrete topology.
\end{prop}

\begin{prop}\label{prop:local}
If $\dim V\geq 2$ and if $G$ is a connected Lie group locally isomorphic to $H(V)/Z\rtimes S$, then $G_\dd$ is not weakly amenable.
\end{prop}
\begin{proof}
Let $\tilde G$ be the universal cover of $G$. Then $G = \tilde G/D$ for some discrete central subgroup $D$ of $\tilde G$. By \eqref{eq:mod-central}, it is enough to prove that $\tilde G_\dd$ is not weakly amenable, and hence we may (and will) assume that $G$ is simply connected.

If $S = \SU(2)$, then the group $H(V)/Z\rtimes S$ is simply connected, so $\tilde G = H(V)/Z\rtimes S$ and we apply Proposition~\ref{prop:minimal-groups}. If $S = \SL(2,\R)$ and $\tilde S = \tilde\SL(2,\R)$, then $H(V)/Z\rtimes\tilde S$ is simply connected, so $\tilde G = H(V)/Z\rtimes\tilde S$ and we apply Proposition~\ref{prop:SL2-covering}.
\end{proof}

For now, let $S = \SL(2,\R)$ and $V = V_n$. If $n = 2m+1$ is odd, the space $\Alt_{\mf s}(V)^*$ is trivial and
\[
H(V)\rtimes S = \R^{2m+1}\rtimes\SL(2,\R).
\]
If $n=2m$ is even, the space $\Alt_{\mf s}(V)^*$ is one dimensional and $H(V)$ is the $2m+1$ dimensional real Heisenberg group $H_{2m+1}$. If $Z = \Alt_{\mf s}(V)^*$, then
\[
H(V)\rtimes S = H_{2m+1}\rtimes\SL(2,\R), \qquad
H(V)/Z \rtimes S = \R^{2m}\rtimes\SL(2,\R).
\]
When $S = \SU(2)$ and $V = V_{2m+1}$, the space $\Alt_{\mf s}(V)^*$ is trivial, and with the notation of \cite{cornulier-jlt} we have
\[
H(V)\rtimes S = D^\R_{2m+1}\rtimes\SU(2).
\]
When $S = \SU(2)$ and $V = V_{4m}$, the space $\Alt_{\mf s}(V)^*$ is three dimensional. If $Z_i\subseteq\Alt_{\mf s}(V)^*$ is a subspace of dimension $3-i$, then with the notation of \cite{cornulier-jlt} we have
\[
H(V)/Z_i\rtimes S = HU^i_{4m}\rtimes\SU(2).
\]

Using the perhaps more iluminating description of the groups $H(V)/Z\rtimes S$ just given, Proposition~\ref{prop:local} translates as

\begin{prop}\label{prop:four-cases}
Let $G$ be a connected Lie group locally isomorphic to one of the following groups:
\begin{itemize}
	\item $D^\R_{2n+1}\rtimes\SU(2)$ for some $i=0,1,2,3$ and some $n\geq 1$;
	\item $HU^i_{4n}\rtimes\SU(2)$ for some $i=0,1,2,3$ and some $n\geq 1$;
	\item $\R^n\rtimes\SL(2,\R)$ for some $n\geq 2$;
	\item $H_{2n+1}\rtimes\SL(2,\R)$ for some $n\geq 1$.
\end{itemize}
Then $G_\dd$ is not weakly amenable.
\end{prop}

\begin{prop}\label{prop:not-central}
Let $G$ be a connected Lie group, and let $G = RS$ be a Levi decomposition (see Section~\ref{sec:Levi}), where $R$ is the solvable radical and $S$ is a semisimple Levi factor. If $[R,S]\neq 1$, then $G_\dd$ is not weakly amenable.
\end{prop}
\begin{proof}
This follows basically from structure theory of Lie algebras together with Proposition~\ref{prop:four-cases}. Indeed, from the assumption $[R,S]\neq 1$ it follows from \cite[Proposition~3.4]{cornulier-jlt} and \cite[Proposition~3.8]{cornulier-jlt} that $G$ contains a connected Lie subgroup $H$ locally isomorphic to one of the groups listed in Proposition~\ref{prop:four-cases}. Since $H_\dd$ is not weakly amenable, $G_\dd$ is not weakly amenable.
\end{proof}

\begin{thm}\label{thm:discrete}
Let $G$ be a connected Lie group, and let $G_\dd$ denote the group $G$ equipped with the discrete topology. The following are equivalent.
\begin{enumerate}
	\item[(1)] $G$ is locally isomorphic to $R\times\SO(3)^a\times\SL(2,\R)^b\times\SL(2,\C)^c$, for a solvable connected Lie group $R$ and integers $a,b,c$.
	\item[(2)] $G_\dd$ is weakly amenable with constant 1.
	\item[(3)] $G_\dd$ is weakly amenable.
	\item[(4)] Every countable subgroup of $G$ is weakly amenable with constant $1$.
	\item[(5)] Every countable subgroup of $G$ is weakly amenable.
\end{enumerate}
\end{thm}
\begin{proof}
Let $G =RS$ be a Levi decomposition of $G$ (see Section~\ref{sec:Levi}).

(1)$\implies$(2): If $S$ is a semisimple Levi factor in $G$, then by assumption $S$ is normal in $G$, and the group $G_d/S_d$ is solvable, since it is a quotient of the solvable group $R_d$. By \eqref{eq:co-amenable}, it is enough to show that $S_\dd$ is weakly amenable with constant 1.

Using \eqref{eq:mod-central}, we may assume that the center of $S$ is trivial. Then $S$ is direct product of factors $\SO(3)$, $\PSL(2,\R)$, and $\PSL(2,\C)$. An application of \eqref{eq:product} and Theorem~\ref{thm:simple-discrete} shows that $S_\dd$ is weakly amenable with constant 1.



(2)$\implies$(4): This is clear.

(4)$\implies$(5): This is clear.

(5)$\implies$(3): This is Lemma~\ref{lem:countable}.

(3)$\implies$(1): Suppose $G$ does not satisfy (1). If $[R,S]\neq 1$, then Proposition~\ref{prop:not-central} shows that $G_\dd$ is not weakly amenable. Otherwise $[R,S] = 1$ and $S$ contains a simple Lie subgroup not locally isomorphic to $\SO(3)$, $\SL(2,\R)$, or $\SL(2,\C)$. It then follows from Theorem~\ref{thm:simple-discrete} that $S_\dd$ is not weakly amenable, and hence $G_\dd$ too is not weakly amenable.
%
\end{proof}

\section{Weak amenability of Lie groups with and without topology}

As a consequence of Theorem~\ref{thm:discrete}, we can answer (part of) Question~1.8 in \cite{MR3345044} affirmatively in the case of connected Lie groups. Indeed, we show below that, for a connected Lie group $G$, if $G_\dd$ is weakly amenable then $G$ too is weakly amenable. We first establish a few lemmas.

\begin{lem}\label{lem:cocompact}
Let $m,n$ be non-negative integers, and let $D\subseteq \R^m\times \Z^n$ be a discrete subgroup. There is a discrete subgroup $D'\subseteq \R^m\times \Z^n$ such that $D\subseteq D'$ and $D'$ is cocompact in $\R^m\times \Z^n$.
\end{lem}
\begin{proof}
Our proof is an application of the characterization of compactly generated, locally compact abelian groups (see \cite[Theorem~9.8]{MR551496}). As $\R^m\times\Z^n$ is compactly generated, so is the quotient $(\R^m\times\Z^n)/D$. Therefore the quotient is of the form $\R^a \times \Z^b \times C$, where $a$ and $b$ are integers and $C$ is a compact abelian group. Clearly, $\R^a \times \Z^b \times C$ has a cocompact discrete subgroup, $\Z^a\times\Z^b\times\{0\}$, and its preimage in $\R^m\times\Z^n$ is a discrete, cocompact subgroup which contains $D$.
\end{proof}


Our next lemma establishes the existence of lattices in certain Lie groups. There are well-known results of Malcev about existence of lattices in nilpotent Lie groups and of Borel about existence of lattices in semisimple Lie groups  (see \cite[Theorem~2.12]{MR0507234} and \cite[Theorem~14.1]{MR0507234}). However, we are interested in some intermediate cases such as the following example, which we have included to give the reader an intuition about the succeeding proof.

\begin{exam}
Fix an irrational number $\theta$. Let $H$ be the universal covering group of $\SL(2,\R)$. Its center is infinite cyclic, and we let $z$ denote a generator of the center of $H$. Consider the group $D = \{( -m - n\theta, z^m, z^n) \mid m,n\in\Z\}$ which is central in $\R\times H\times H$, and let $G$ be the quotient group $G = (\R\times H\times H)/D$. We will describe a lattice in $G$.

The group $\SL(2,\R)$ admits a lattice $F$ isomorphic to the free group on two generators. By freeness, $F$ lifts to a subgroup $\tilde F$ of $H$. Then $(\Z\times \tilde F\times\tilde F)D$ is a lattice in $\R\times H\times H$, and it obviously contains $D$, so it factors down to a lattice in $(\R\times H\times H\times)/D$.
\end{exam}

\begin{lem}\label{lem:existence-of-lattice}
A connected Lie group locally isomorphic to
\[
\R^m\times\SL(2,\R)^n,
\]
where $m$ and $n$ are non-negative integers, contains a lattice.
\end{lem}
\begin{proof}
We first introduce some notation. For any Lie group $L$, let $Z(L)$ denote the center of $L$. We use $1$ to denote the neutral element (or $0$ for the group $\R$). Let $H$ be the universal covering group of $\SL(2,\R)$. Its center $Z(H)$ is infinite cyclic. 

Set $\tilde G = \R^m\times H^n$. Then $\tilde G$ is a simply connected and connected Lie group, and any connected Lie group $G$ locally isomorphic to $\R^m\times\SL(2,\R)^n$ is of the form $G=\tilde G / D$ for some discrete central subgroup $D$ of $\tilde G$. Let $\pi\colon \tilde G\to G$ be the quotient homomorphism $\pi(x) = xD$.

Suppose $D\subseteq D'$ for some other discrete central subgroup $D'$ in $\tilde G$ and that $\tilde G/D'$ contains a lattice. Then the preimage under $\tilde G/D \to \tilde G/D'$ of any lattice in $\tilde G/D'$ is a lattice in $\tilde G/D$. The center of $\tilde G$ is $Z(\tilde G) = \R^m \times Z(H)^n \simeq \R^m\times\Z^n$, so by Lemma~\ref{lem:cocompact} we may without loss of generality suppose that $D$ is discrete and cocompact in $Z(\tilde G)$.

The quotient $H/Z(H)$ is $\PSL(2,\R)$, and it is well-known that $\PSL(2,\R)$ has a lattice $F$ isomorphic to a free group on two generators (see e.g. \cite[Example~B.2.5(iv)]{MR2415834}). By the universal property of free groups, there is a subgroup $\tilde F\subseteq H$ such that the quotient map $H\to\PSL(2,\R)$ maps $\tilde F$ bijectively onto $F$. The preimage of $F$ in $H$ is $ \tilde F Z(H)$, which is a lattice in $H$. Also, $\tilde F\cap Z(H) = \{1\}$.

Consider the subgroup $\Gamma = \{0\}^m \times \tilde F^n$ in $\tilde G$. We will show that $\pi(\Gamma) = \Gamma D/D$ is a lattice in $G$. 


The group $\Gamma D$ is discrete in $\tilde G$: Since it is a countable subgroup, it is enough to see that $\Gamma D$ is closed in $\tilde G$. Now, $\Gamma D$ is clearly closed in $\Gamma Z(\tilde G)$, which is closed in $\tilde G$ since $\Gamma Z(\tilde G) = \R^m \times ( \tilde F Z(H))^n$.

The group $\pi(\Gamma)$ is discrete in $G$: As $\pi$ is an open map, $\pi(W)$ is an open set in $G$ and $\pi(\Gamma) \cap \pi(W) = \{1\}$. Indeed, if $w\in W$ and $\gamma\in\Gamma$ satisfy $\pi(w) = \pi(\gamma)$, then it follows that $w\in \Gamma D$ so $w = 1$. Thus, $\pi(\Gamma)$ is discrete in $G$.

The group $\pi(\Gamma)$ has finite covolume in $G$: Let $\psi\colon G\to G/Z(G)$ be the quotient homomorphism. As $D$ is discrete and $\tilde G$ is connected, $Z(G) = Z(\tilde G)/ D$. We thus have isomorphisms
\[
G/Z(G) \simeq \tilde G/Z(\tilde G) \simeq \PSL(2,\R)^n,
\]
and under these isomorphisms $\psi\pi(\Gamma) = F^n$. As $F$ is a lattice in $\PSL(2,\R)$, there is a Borel set (even a Borel fundamental domain) $\Omega\subseteq G/Z(G)$ of finite measure such that $\Omega (\psi\pi(\Gamma)) = G/Z(G)$ (see \cite[Proposition~B.2.4]{MR2415834}). By outer regularity, we may assume that $\Omega$ is in addition open (but no longer a fundamental domain). The inverse image $\psi^{-1}(\Omega) \subseteq G$ is then also open and $\psi^{-1}(\Omega)\pi(\Gamma) = G$. As $\psi^{-1}(\Omega)$ is open, its characteristic function is lower semicontinuous, and it follows from Weil's integration formula for lower semicontinuous functions (see \cite[(3.3.13)]{MR1802924}) that the Haar measure of $\psi^{-1}(\Omega)$ is the Haar measure of $\Omega$ multiplied by the Haar measure of $Z(G)$. As $D$ is cocompact in $Z(\tilde G)$, the center $Z(G) = Z(\tilde G) / D$ is compact. Therefore $Z(G)$ has finite Haar measure, and in conclusion $\psi^{-1}(\Omega)$ has finite Haar measure.

By \cite[Proposition~B.2.4]{MR2415834} it follows that $\pi(\Gamma)$ is a lattice in $G$, and this completes the proof.
\end{proof}

\begin{thm}\label{thm:non-discrete}
Let $G$ be a connected Lie group locally isomorphic to
\[
G\approx R\times\SO(3)^a\times\SL(2,\R)^b\times\SL(2,\C)^c,
\]
for a solvable connected Lie group $R$ and integers $a,b,c$. Then $G$ is weakly amenable with constant 1, i.e., $\Lambda_\WA(G) = 1$.
\end{thm}
\begin{proof}
The strategy of the proof is to reduce the problem to the case where $G$ is locally isomorphic to the group appearing in Lemma~\ref{lem:existence-of-lattice}. This is done in several steps. We first show how to get rid of the factors $\SO(3)$ and $\SL(2,\C)$. Then we show how to reduce the radical to an abelian group.

Let $\mf g$ be the Lie algebra of $G$. With $\mf r$ the solvable radical in $\mf g$, we have (recall $\mf{so}(3) = \mf{su}(2)$)
\[
\mf g = \mf r \oplus \mf{su}(2)^a \oplus \mf{sl}(2,\R)^b \oplus \mf{sl}(2,\C)^c.
\]
Set $\mf s_0 = \mf{su}(2)^a \oplus \mf{sl}(2,\C)^c \subseteq \mf g$ and let $S_0$ be the connected semisimple Lie subgroup of $G$ associated with $\mf s_0$. Note that the center $Z(S_0)$ of $S_0$ is finite, since the simply connected group $\SU(2)^a\times\SL(2,\C)^c$ has finite center.
%
%

Set $\mf h = \mf r \oplus \mf{sl}(2,\R)^b$ so that $\mf g = \mf h\oplus\mf s_0$, and let $H$ be the connected Lie subgroup of $G$ associated with $\mf h$. As $[\mf h,\mf s_0] = 0$, the subgroups $H$ and $S_0$ commute, and the multiplication map $\phi\colon H\times S_0\to G$ is a homomorphism. The image $\phi(H\times S_0)$ is a connected Lie subgroup of $G$ containing both $H$ and $S_0$. It follows that $\phi$ is surjective and $G\simeq (H\times S_0)/\ker\phi$.

The kernel $\ker\phi$ is precisely
\[
\ker\phi = \{(h,h^{-1}) \mid h\in H\cap S_0\}.
\]
Since $H$ and $S_0$ commute, the group $H\cap S_0$ is central in $S_0$ and hence finite. Then $\ker\phi$ is also a finite group. By \eqref{eq:mod-compact} and \eqref{eq:product} we have
\[
\Lambda_\WA(G) = \Lambda_\WA(H\times S_0) = \Lambda_\WA(H)\Lambda_\WA(S_0).
\]
Note that $\Lambda_\WA(S_0) = 1$ by Theorem~\ref{thm:semisimple}, since by Theorem~\ref{thm:simple} both $\SU(2)$ and $\SL(2,\C)$ are weakly amenable with constant $1$ (recall $\SL(2,\C) \approx \SO(1,3)$).

We have thus reduced the problem to the case where $G = H$ is locally isomorphic to $R \times \SL(2,\R)^b$. Let $G = RS$ be a Levi decomposition of $G$. Then the closure $\bar S$ of $S$ in $G$ is a closed connected normal subgroup of $G$ whose Lie algebra is a subalgebra of $\mf g$, and the quotient $G/\bar S$ is solvable. By \eqref{eq:co-amenable}, it suffices to prove that $\bar S$ is weakly amenable with constant 1. We may thus suppose that $S$ is a dense connected Lie subgroup of $G$.

When $S$ is dense, a theorem of Mostow \cite[\S~6]{MR0048464} shows that $G$ is of the form $G = SC$ where $C$ is a connected Lie subgroup of the center of $G$ and in the closure of the center of $S$. It follows that the solvable radical is abelian and $G$ is locally isomorphic to $\R^n\times \SL(2,\R)^b$ for some integer $n$.

If $G$ is simply connected, then $G = \R^n\times\tilde\SL(2,\R)^b$, where $\tilde\SL(2,\R)$ denotes the universal covering group of $\SL(2,\R)$. The fact that $\Lambda_\WA(G) = 1$ is basically \cite{MR1079871} (recall $\tilde\SL(2,\R) = \tilde\SU(1,1)$) together with the product formula \eqref{eq:product} (see also Theorem~\ref{thm:simple}).

The general case can then be deduced from the simply connected case as follows. Let $\tilde G$ be the universal covering group of $G$. By Lemma~\ref{lem:existence-of-lattice}, there is a lattice $\Gamma$ in $G$. Let $\tilde\Gamma$ be the preimage of $\Gamma$ in $\tilde G$ under the covering homomorphism $\tilde G\to G$. Then $\tilde\Gamma$ is a lattice in $\tilde G$, and $\tilde\Gamma$ is a central extension of $\Gamma$. By \eqref{eq:lattice} and \eqref{eq:mod-central} we have
\[
\Lambda_\WA(G) = \Lambda_\WA(\Gamma) \leq \Lambda_\WA(\tilde\Gamma) = \Lambda_\WA(\tilde G) = 1.
\]
This shows that $\Lambda_\WA(G) = 1$, and the proof is complete.
\end{proof}

\begin{cor}\label{cor:question18}
Let $G$ be a connected Lie group. If $G_\dd$ is weakly amenable, then $G$ is weakly amenable. In this case, $\Lambda_\WA(G_\dd) = \Lambda_\WA(G) = 1$.
\end{cor}
\begin{proof}
This is immediate from Theorem~\ref{thm:discrete} and Theorem~\ref{thm:non-discrete}.
\end{proof}

\section*{Acknowledgements}
We would like to thank the anonymous referee for helpful suggestions and remarks.


\end{document}